\documentclass[11pt]{article}
\usepackage{amsmath, amssymb, amsthm, amscd, amsfonts, enumerate, mathrsfs}

\newtheorem{theorem}{Theorem}[section]
\newtheorem{proposition}[theorem]{Proposition}
\newtheorem{lemma}[theorem]{Lemma}
\newtheorem{corollary}[theorem]{Corollary}

\theoremstyle{definition}
\newtheorem{definition}[theorem]{Definition}

\newtheorem{example}[theorem]{Example}

\newtheorem{remark}[theorem]{Remark}

\newcommand{\N}{\mathbb{N}}		
\newcommand{\Z}{\mathbb{Z}}		
\newcommand{\C}{\mathbb{C}}		
\newcommand{\A}{\mathcal{A}}		
\newcommand{\B}{\mathcal{B}}		
\renewcommand{\O}{\mathcal{O}}	
\renewcommand{\S}{\mathcal{S}}		
\renewcommand{\H}{\mathcal{H}}		
\newcommand{\M}{\mathcal{M}}		
\renewcommand{\a}{\mathbf{a}}		
\renewcommand{\b}{\mathbf{b}}		

\newcommand{\abs}[1]{\left\lvert #1 \right\rvert}
\newcommand{\norm}[1]{\left\lVert #1 \right\rVert}
\newcommand{\innpr}[2]{\left\langle #1, #2 \right\rangle}

\DeclareMathOperator{\id}{id}
\DeclareMathOperator{\orb}{orb}
\DeclareMathOperator{\Orb}{Orb}
\DeclareMathOperator{\dist}{d}
\DeclareMathOperator{\Span}{span}

\title{Dynamical systems with bounded condition and $C^{*}$-algebras}

\author{Takehiko Mori\footnote{ORCID:0000-0001-6448-8293}\\
Graduate School of Science and Engineering\\
Chiba University \\
Inage-ku, Chiba 263-8522, Japan
}

\date{}

\begin{document}


\maketitle

\begin{abstract}
In this paper, we study abstract dynamical systems with discrete phase spaces. 
One example of such a system is induced by the $3 x{+}1$-map on the set of all natural numbers, also known as the Collatz map. 
Our main focus is on dynamical systems induced by maps on countable discrete sets that satisfy a bounded condition. 
When these maps satisfy the bounded and a  separating conditions, a minimality of the induced dynamical systems is equivalent to the irreducibility of certain $C^{*}$-algebras on certain Hilbert spaces. 
For a map $f$ on a general discrete phase space, we consider $f$-invariant sets and investigate their properties. 
When the phase space is countable and the map satisfies the bounded condition, we construct an order-preserving injection from the family of $f$-invariant sets to the family of reducing subspaces for the corresponding $C^{*}$-algebra. 
By introducing the totally uniqueness condition for $f$, we show that this injection is a bijection if $f$ satisfies this condition. 
This condition is crucial in providing a symbolic representation of the dynamical system induced by $f$, and we discuss the relationship between this symbolic representation and that of a topological dynamical system. 
\end{abstract}


\maketitle

\section*{Introduction}

The interaction between dynamical systems and operator theory is an a long-standing and active area of study. 
In particular, a great deal of interesting work has been carried out on linking topological dynamical systems to the theory of $C^{*}$-algebras. 
A topological dynamical system with discrete time consists of a metric space $X$ and a continuous map $f$ on $X$. 
Typically, $X$ is a compact metric space and $f$ is a homeomorphism.
A $C^{*}$-algebra is a $*$-algebra consisting of bounded linear operators on a Hilbert space that is given the topology of the operator norm. 
In their pioneering work \cite{GPS95JRAM}, T. Giordano, I. F. Putnam and C. F. Skau proved that two minimal homeomorphisms on the Cantor set are strongly orbit equivalent if and only if their associated crossed-product $C^{*}$-algebras are isomorphic. 
Giordano, H. Matui, Putnam and Skau later generalized this result to minimal $\Z^{d}$-actions in \cite{GMPS10IM}. 
M. Boyle and J. Tomiyama have studied the relationship between orbit equivalence for topologically free homeomorphisms on compact Hausdorff spaces and their associated crossed product $C^{*}$-algebras (see \cite{BT98JMSJ}). 

Another important example is the Cuntz-Krieger algebra, which is a universal $C^{*}$-algebra that arises from a topological Markov shift. 
A universal $C^{*}$-algebra is defined by its generators and relations. 
Further details about universal $C^{*}$-algebras can be found in \cite[Section 1]{Bl85MS}, for example. 
The Cuntz-Krieger algebra, denoted by $\O_{A}$, is a well-known example of a universal $C^{*}$-algebra. 
It is determined by a set of partial isometries that satisfy the Cuntz-Krieger relation, which can be described as follows. 

\begin{definition}[\cite{CK80IM}, The Cuntz-Krieger algebras $\O_{A}$]
For $k \in \N_{>1}$ and $k \times k$ matrix $A=(A(i,j))$ with entries in $\{0,1\}$ such that no row is zero, the Cuntz-Krieger algebra $\O_{A}$ is the universal $C^{*}$-algebra generated by partial isometries $S_{1},S_{2}, \cdots, S_{k}$ which satisfy 

\begin{enumerate}[(i)]
\item $\sum_{i=1}^{k}S_{i}S_{i}^{*}=I$; 
\item $S_{j}^{*}S_{j}=\sum_{i=1}^{k}A(j,i)S_{i}S_{i}^{*}$. 
\end{enumerate}

\end{definition}

In \cite{Cu77CMP}, J. Cuntz studied $\O_{A}$ when $A(i,j)=1$ for every $1 \leq i,j \leq k$. 
It is called the Cuntz algebra denoted by $\O_{k}$. 
In \cite{CK80IM}, Cuntz and W. Krieger discuss a connection between the properties of dynamical systems related to the idea of Markov partitions for topological dynamical systems and their associated $C^{*}$-algebras. 
For more general subshifts, K. Matsumoto introduced the $C^{*}$-algebras and computed their $K$-groups in \cite{Ma97IJM} and \cite{Ma98MS}. 
Matsumoto and Matui classified topological Markov shifts up to continuous orbit equivalence based on the concept of the Cuntz-Krieger algebras in \cite{MM14KJM}. 

The above results were established derived from the theory of topological dynamical systems with compact, metric phase spaces. 
However, we shall be interested in dynamical systems whose phase spaces are discrete - that is to say, those that do not have a topology. 
One example of such a system is the pair of the set of all natural numbers and $3 x{+}1$-map, also known as the Collatz map. 
The Collatz map $f$ is defined for positive integers by setting $f(n)$ equal to $3 n{+}1$ when $n$ is odd and to $n/2$ when $n$ is even. 
The Collatz conjecture, also known as the $3 n{+}1$-problem, concerns positive integers and is named after L. Collatz, who first described it in an informal lecture at the International Mathematical Congress in Cambridge, Massachusetts, in 1950 \cite{La10TUC}. 
The conjecture states that, starting from any positive integer $n$, some iteration of $f$ will eventually yield the value $1$. 
We regard $(\N,f)$ as a dynamical system with discrete phase space.
In this case, the phase space is countable. 
But, the dynamical system is not simple, since the Collatz conjecture remains unsolved for a long time. 
There are a few researches which study the Collatz conjecture by functional analytic approaches in this point of view, e.g., \cite{LSW99EM} and \cite{LP21IFACPOL}. 

In \cite{Mo25AOT}, we presented a method of formulating the Collatz conjecture in terms of dynamical systems and operator theory. 
We constructed some $C^{*}$-algebras associated with the dynamical systems induced by the Collatz map. 

\begin{theorem}[{\cite[Theorem 4.2.7]{Mo25AOT}}] \label{IrrEquCol}
Let $f$ be the Collatz map, $\H$ be a Hilbert space with $\dim\H=\aleph_{0}$ and $\{e_{n}\}_{n \in \N}$ be a complete orthonormal system for $\H$. 
When $T_{1},T_{2} \in \B(\H)$ are defined by: 

\begin{equation*}
T_{1}e_{n}= 
\begin{cases}
e_{3n+1},	&	n:\text{odd}, \\ 
0,		&	n:\text{even}, 
\end{cases}
\end{equation*}

\begin{equation*}
T_{2}e_{n}= 
\begin{cases}
0,		&	n:\text{odd}, \\ 
e_{n/2}, 	&	n:\text{even}, 
\end{cases}
\end{equation*} 

where $n \in \N$, the following are equivalent: 

\begin{enumerate}[(i)]
\item	$C^{*}(T_{1}, T_{2})$ is irreducible; 
\item	The Collatz conjecture holds. 
\end{enumerate}

\end{theorem}

$\B(\H)$ denotes the set of all bounded linear operators on $\H$. 
A subspace $\M \subseteq \H$ is said to be a reducing subspace for a set $\S \subseteq \B(\H)$ when $T\M, T^{*}\M \subseteq \M$ for every $T \in \S$. 
In this paper, we will only refer to a subspace as `reducing' if it is a topologically closed. 
A reducing subspace is called trivial if it is either $\{0\}$ or $\H$. 
For a $C^{*}$-algebra $\A \subseteq \B(\H)$, we say that $\A$ is irreducible when there are no non-trivial reducing subspaces for $\A$. 

The first-return map for a discrete dynamical system is defined in the same way as for a continuous dynamical system. 

\begin{theorem}[{\cite[Proposition 4.3.4, Theorem 4.3.10]{Mo25AOT}}]
Let $f$ be the Collatz map and $\H$ be a Hilbert space with $\dim\H=\aleph_{0}$. 
Let $N_{1} = \{n \in \N \mid n \equiv 1, 5 \pmod{6}\}$, $N_{2} = \{n \in \N \mid n \equiv 4, 16 \pmod{18}\}$, $P$ be the first-return map for $f$ on $N_{1} \cup N_{2}$, and $\{e_{n}\}_{n \in N_{1} \cup N_{2}}$ be a complete orthonormal system for $\H$. 
When $T_{1}, T_{2} \in \B(\H)$ is defined by: 

\begin{equation*}
T_{1}e_{n}= 
\begin{cases}
e_{P(n)},	&	n:\text{odd}, \\ 
0,		&	n:\text{even}, 
\end{cases}
\end{equation*}

\begin{equation*}
T_{2}e_{n}= 
\begin{cases}
0,		& n:\text{odd}, \\ 
e_{P(n)}, 	& n:\text{even}, 
\end{cases}
\end{equation*}

where $n \in N_{1} \cup N_{2}$, $C^{*}(T_{1}, T_{2})$ is isomorphic to the Cuntz algebra $\O_{2}$ and the following are equivalent: 

\begin{enumerate}[(i)]
\item	$C^{*}(T_{1}, T_{2})$ is irreducible; 
\item	The Collatz conjecture holds. 
\end{enumerate}

\end{theorem}

In this paper, we establish a new connection between dynamical systems with discrete phase spaces and $C^{*}$-algebras. 
Invariant sets for a map on a discrete phase space are given and we described their properties. 
When a phase space is countable and a map on the space satisfies a bounded condition, we construct a $C^{*}$-algebra on a Hilbert space which has a countable complete orthonormal system. 
By considering the inclusion orders, we give an order-preserving injection from the family of invariant sets for the map to the family of reducing subspaces for the corresponding $C^{*}$-algebra. 
Furthermore, we provide a sufficient condition for that the above injection is a bijection. 
The condition will be called the totally uniqueness condition for the map. 
We also prove that every $q x{+}d$-function, where $q$ and $d$ are arbitrary positive odd integers, including the Collatz map and more general maps satisfy the totally uniqueness condition. 
Through this condition, we discuss a method for relating dynamical systems on discrete phase spaces to symbolic dynamical systems. 
We also examine the connection between this symbolic representation and that for topological dynamical systems. 

The paper is organized as follows. 
In section 1, we provide the concept of invariant sets in dynamical systems with discrete phase spaces and minimality of the dynamical systems. 
In section 2, we give constructions of $C^{*}$-algebras for maps on discrete phase spaces satisfying the bounded condition and theorems revealed in previous studies. 
We show that, for such a dynamical system, there exists an order-preserving injection from the family of invariant sets for the map defined on the phase space to the family of reducing subspaces for the associated $C^{*}$-algebra. 
Moreover, if Markov partitions for such dynamical systems exist, the associated $C^{*}$-algebras are naturally related to the Cuntz-Krieger algebras. 
In section 3, we define the uniqueness condition for maps. 
Under the condition, we analyze the property of a loop of dynamical systems. 
It is then shown that the tuple of parities of a loop for the Collatz map is aperiodic. 
In section 4, we introduce the totally uniqueness condition for maps, a stronger condition than the uniqueness condition. 
We show that the condition is a sufficient condition for that the order-preserving map in section 2 is a bijection. 
One example that satisfies the totally uniqueness condition is the Collatz map. 
For more general maps, we prove that they satisfy the condition. 
In section 5, we discuss a symbolic representation of dynamical systems with discrete phase spaces. 
From the totally uniqueness condition, we can naturally construct a faithful symbolic representation of dynamical systems with a discrete phase space. 
We verify basic properties of the symbolic representation of our dynamical system by comparing it with that of a topological dynamical system. 

Throughout this paper, $\N$ and $\C$ denote the sets of all positive integers and complex numbers, respectively. 
The cardinality of $\N$ is denoted by $\aleph_{0}$. 
The absolute value and complex conjugate of a number $x$ are denoted by $\abs{x}$ and $\overline{x}$. 
The cardinality of any set $X$ is denoted by $|X|$, and the identity map on $X$ is denoted by $\id_{X}$. 
If $X$ is a topological space and $Y$ is a subspace of $X$, then the closure of $Y$ in $X$ is denoted by $\overline{Y}$. 
For $k \in \N_{>1}$, the disjoint union over $m \in \N$ of $\{1,2,\cdots,k\}^{m}$ is denoted by $\{1,2,\cdots,k\}^{*}$.


\section{Dynamical systems with discrete phase spaces}


In this section, we will define invariant sets in dynamical systems with discrete phase spaces. 
A discrete-time dynamical system consists of a phase space $X$ and a map $f$ on $X$. 
Typically, $X$ is a differentiable manifold and $f$ is a diffeomorphism, or $X$ is a compact metric space and $f$ is a homeomorphism. 
Here, we are interested in cases where $X$ is a discrete set and $f$ is a map that does not necessarily have to be invertible or surjective. 
For now, let $X$ be a discrete set, i.e. a set that does not require a topology to exist. 

We denote by $f^{m}$ the $m$th iterate of $f$, where $m \in \{0\} \cup \N$, i.e., $f^{0} = \id_{X}$ and $f^{m+1} = f \circ f^{m}$. 
The (forward) orbit of a point $x \in X$, denoted by $\orb(x;f)$, is defined as follows: 

\begin{equation*}
\orb(x;f)=\{f^{n}(x) \mid n \in \{0\} \cup \N\}. 
\end{equation*}

For convenience, we define an equivalence relation on $X$ by $f$. 
It is defined as follows: 
For $x, y \in X$, 

\begin{equation*}
x \sim y \iff \orb(x; f) \cap \orb(y; f) \neq \emptyset, 
\end{equation*}

in other words, there exist $l,m \in \N$ such that $f^{l}(x) = f^{m}(y)$. 
We define the total orbit of a point of $X$ using the above equivalence relation. 

\begin{definition}[Total orbits]
For $x \in X$, the total orbit of $x$, denoted by $\Orb(x;f)$, is defined as 

\begin{equation*}
\Orb(x;f)=\{y \in X \mid x \sim y\}. 
\end{equation*}

\end{definition}

Invariant sets for dynamical systems with discrete phase spaces are defined as follows: 

\begin{definition}[Invariant sets]
Let $K$ be a subset of $X$. 
If $f(K) \subseteq K$ and $f^{-1}(K) \subseteq K$, then we say that $K$ is an $f$-invariant set. 
An $f$-invariant set $K$ is called trivial if $K=\emptyset$ or $K=X$. 
We say that the pair $(X,f)$ is minimal when there are no non-trivial $f$-invariant sets. 
\end{definition}

If $f$ is invertible, then $K \subseteq X$ is an $f$-invariant set if and only if $f(K)=K$. 
The total orbit is the smallest $f$-invariant set. 

\begin{lemma}

\begin{enumerate}[(i)]
\item	For $x \in X$, $\Orb(x;f)$ is a $f$-invariant set. If $K \subseteq X$ is a $f$-invariant set containing $x$, then $\Orb(x;f) \subseteq K$. 
\item	$(X,f)$ is minimal if and only if $x \sim y$ for every $x,y \in X$. 
\end{enumerate}

\end{lemma}

\begin{proof}
Let $y \in \Orb(x;f)$ and assume that $z \in f^{-1}(y)$. 
Since $f(y),z \sim y$ and $\sim$ is an equivalence relation on $X$, we have that $f(y),z \sim x$. 
Therefore, $f(y),z \in \Orb(x;f)$. 
It follows that $f(\Orb(x;f)), f^{-1}(\Orb(x;f)) \subseteq \Orb(x;f)$, and so $\Orb(x;f)$ is a $f$-invariant set. 

Suppose that $K \subseteq X$ is a $f$-invariant set containing $x$. 
If $y \in \Orb(x;f)$, then there exist $l,m \in \N$ such that $f^{l}(x)=f^{m}(y)$. 
This implies that $y \in f^{-m}(f^{l}(x)) \subseteq f^{-m}(f^{l}(K)) \subseteq K$. 
Hence, $\Orb(x;f) \subseteq K$. 
Therefore,  the smallest $f$-invariant set containing $x$ is equal to $\Orb(x;f)$. 
Thus, $(i)$ holds. 

According to the above discussion, if $(X,f)$ is minimal, then $\Orb(x;f)=X$ for every $x \in X$. 
This implies that $x \sim y$ for every $x,y \in X$. 
Conversely, assume that $x \sim y$ for every $x,y \in X$. 
Then, $X=\Orb(x;f)$ for every $x \in X$. 
Since $\Orb(x;f)$ is the smallest $f$-invariant set containing $x$, the minimality of $(X,f)$ holds. 
Hence, $(ii)$ holds. 
\end{proof}

The definition of a total orbit is the same as for a standard topological dynamical system. 
Although $f$ does not need to be injective, we will mostly assume that $f^{-1}(x)$ is a finite set for all $x \in X$. 
In this case, the total orbit of each point becomes a countable set. 
In the context of a topological dynamical system, it is referred to as `minimal' if and only if the total orbit of any point is dense in the phase space. 
However, the phase space of our dynamical system $(X,f)$ is discrete, so the above definition of minimality does not apply. 
For the system to be minimal, $X$ must necessarily be countable. 
For the rest of this paper, we will assume that $X$ is a countable set.

\section{Bounded condition and $C^{*}$-algebras}

In \cite{Mo25AOT}, we presented the definitions of the bounded and separating conditions for $f$, alongside the following results for the case $|X|=\aleph_{0}$. 
We note that most of these results also hold when $|X| < \aleph_{0}$. 
The bounded condition plays an crucial role in this study. 
Consider a partition of the discrete space $X$. 

\begin{definition}[{\cite[Definition 5.1.1]{Mo25AOT}}, Bounded condition]
When there exist mutually disjoint subsets $X_{1},X_{2}, \cdots ,X_{k}$ of $X$, where $k \in \N$ such that: 

\begin{enumerate}[(i)]
\item	$\bigcup_{i=1}^{k} X_{i}=X$; 
\item	$f|_{X_{i}}$ is injective, where $1 \leq i \leq k$, 
\end{enumerate}

we say that $f$ satisfies the bounded condition. 
\end{definition}

In what follows, we will assume that $f$ satisfies the bounded condition. 
For $(X,f)$, we let $\H$ be a Hilbert space over $\C$ with $\dim\H=|X|$ and let $\{e_x\}_{x \in X}$ be a complete orthonormal system (abbreviated as C.O.N.S.) for $\H$. 
Define $k$ partial isometries $T_{1},T_{2},\cdots,T_{k} \in \B(\H)$ by: 

\begin{align*}
T_{i} e_{x}= 
\begin{cases}
e_{f(x)},	&	x \in X_{i}, \\ 
0,		&	x \notin X_{i}, 
\end{cases}
\end{align*}

where $1 \leq i \leq k$. 
The bounded condition is defined as meaning that each $T_{i}$ is a partial isometry. 
This setting leads the following lemma and theorem. 
For detailed proofs, please refer to \cite{Mo25AOT}. 

\begin{lemma}[{\cite[Lemma 5.1.7]{Mo25AOT}}] \label{RedSubOrb}
For every $x \in X$, 

\begin{equation*}
\overline{C^{*}(T_{1},T_{2},\cdots,T_{k})e_{x}}=\overline{\Span\{e_{y} \mid y \in \Orb(x;f)\}}. 
\end{equation*}

\end{lemma}

\begin{theorem}[{\cite[Corollary 5.1.9]{Mo25AOT}}] \label{IrrMin}
If $C^{*}(T_{1},T_{2},\cdots,T_{k})$ is irreducible, then $(X,f)$ is minimal. 
\end{theorem}

According to this theorem, the minimality of every dynamical system with bounded condition is a necessary condition for the irreducibility of the associated $C^{*}$-algebra. 
It is also a sufficient condition when the following separating condition met. 

Let $m \in \N$ and $\sigma$ be a cyclic permutation of $\{1,2,\cdots,m\}$ defined by:  $1 \mapsto 2 \mapsto 3 \mapsto \cdots \mapsto m-1 \mapsto m \mapsto 1$. 
We say that a $m$-tuple of numbers $(i_{1},i_{2},\cdots,i_{m})$ is aperiodic if $(i_{\sigma^{j}(1)},i_{\sigma^{j}(2)},\cdots,i_{\sigma^{j}(m)})\neq (i_{1},i_{2},\cdots,i_{m})$ for every $1 \leq j < m$. 

\begin{definition}[{\cite[Definition 5.2.1]{Mo25AOT}}, Separating condition] \label{SepCon}
Fix $x \in X$. 
When the following conditions hold: 

\begin{enumerate}[(i)]
\item	There exists $m \in \N$ satisfying $f^{m}(x)=x$; 
\item	For the minimum $m$ as above, the $m$-tuple $(i_{1},i_{2},\cdots,i_{m}) \in \{1,2,\cdots,k\}^{*}$, such that $f^{j-1}(x) \in X_{i_{j}}$ where $1 \leq j \leq m$, is aperiodic, 
\end{enumerate}

we say that $f$ satisfies the separating condition for $x$. 
\end{definition}

\begin{theorem}[{\cite[Theorem 5.2.4]{Mo25AOT}}] \label{MinIrr}
Assume that $f$ satisfies the separating condition for some point $x \in X$. 
If $(X,f)$ is minimal, then $C^{*}(T_{1},T_{2},\cdots,T_{k})$ is irreducible. 
\end{theorem}

\begin{example}

\begin{enumerate}[(i)]

\item	(The Collatz map) 
The Collatz map $f:\N \to \N$ is defined by: 

\begin{equation*}
f(n)= 
\begin{cases}
3n+1,	&	n:\text{odd}, \\ 
n/2,		&	n:\text{even}. 
\end{cases}
\end{equation*}

Let $X_{1}=\{n \in \N \mid n:\text{odd}\}$ and $X_{2}=\{n \in \N \mid n:\text{even}\}$. 
Since $X_{1} \cup X_{2}=\N$ and $f|_{X_{1}},f|_{X_{2}}$ are injective, $f$ satisfies the bounded condition. 
It follows that $f^{3}(1)=1, 1 \in X_{1}, f(1)=4 \in X_{2}, f^{2}(1)=2 \in X_{2}$, and $(1,2,2)$ is aperiodic. 
Hence, the Collatz map $f$ satisfies the separating condition for $1$. 

\item	($q x{+}1$-function for $q=\text{Mersenne numbers}$) 
Let $q \in \N$ be an odd number. 
We define the $q x{+}1$-function $f_{q}:\N \to \N$ as follows: 

\begin{equation*}
f_{q}(n)= 
\begin{cases}
q n+1,	&	n:\text{odd}, \\ 
n/2,		&	n:\text{even}. 
\end{cases}
\end{equation*}

By setting $X_{1}=\{n \in \N \mid n:\text{odd}\}, X_{2}=\{n \in \N \mid n:\text{even}\}$, $f_{q}$ satisfies the bounded condition. 

Assume that $q$ is a Mersenne number, i.e., $q=2^{m}-1$ for some $m \in \N$. 
It follows that $1 \in X_{1}, f(1)=2^{m} \in X_{2}, f^{2}(1)=2^{m-1} \in X_{2}, \cdots, f^{m}(1)=2 \in X_{2}$, and $f^{m+1}(1)=1$. 
Since $(1,2,\cdots,2)$ is aperiodic, $f_{q}$ satisfies the separating condition for $1$. 

\item	($3x{+}d$-function) 
Let $d \in \N$ be an odd number. We define the $3 x{+}d$-function $f_{d}:\N \to \N$ as follows: 

\begin{equation*}
f_{d}(n)= 
\begin{cases}
3 n+d,	&	n:\text{odd}, \\ 
n/2,		&	n:\text{even}. 
\end{cases}
\end{equation*}

By setting $X_{1}=\{n \in \N \mid n:\text{odd}\}, X_{2}=\{n \in \N \mid n:\text{even}\}$, $f_{d}$ satisfies the bounded condition. 
It follows that $d \in X_{1}, f_{d}(d)=4d \in X_{2}, f_{d}^{2}(d)=2d \in X_{2}$, and $f_{d}^{3}(d)=d$. 
Since $(1,2,2)$ is aperiodic, $f_{d}$ satisfies the separating condition for $d$. 
\end{enumerate}

\end{example}

We note that the separating condition for $f$ depends on the choice of partition in the bounded condition of $X$. 
When $X$ is a compact metric space and $f$ is homeomorphism on $X$, a Markov partition is defined by using its distance and topology. 
For more detail, see a textbook of general dynamical systems, e.g. \cite{AH94NHML}. 
A Markov partition induces a Markov shift, a symbolic dynamical system of a special kind called a shift of a finite type. 
For a topological Anosov homeomorphism of compact metric space and a Markov partition of the space, there exists a continuous surjection from induced symbolic dynamical system to the original dynamical system which satisfies a commutative diagram (Theorem 4.3.4. in \cite{AH94NHML}). 
The above result is proved by using the compactness of the phase space. 
But, distance and compactness can not apply for dynamical systems with discrete phase spaces. 
Therefore, the concept of a Markov partition of discrete phase spaces cannot be derived directly from the theory of topological dynamical systems. 

Although the bounded condition of $f$ is not directly related to a Markov partition, we have studied a condition which is closely related to a Markov shift and the associated Cuntz-Kriger algebra. 

\begin{definition}[{\cite[Definition 5.3.2]{Mo25AOT}}, Cuntz-Krieger condition]
When the following hold: 

\begin{enumerate}[(i)]
\item	$f$ is surjective; 
\item	For every $1 \leq j \leq k$, there exists $J \subseteq \{1,2,\cdots,k\}$ such that $f(X_{j})=\cup_{i \in J}X_{i}$, 
\end{enumerate}

we say that $f$ satisfies the Cuntz-Krieger condition. 
\end{definition}

If $f$ satisfies the Cuntz-Krieger condition, then there exist the Cuntz-Krieger algebra $\O_{A}$ and a surjective $*$-homomorphism $\phi:\O_{A} \to C^{*}(T_{1},T_{2},\cdots,T_{k})$ (see \cite[Proposition 5.3.3]{Mo25AOT}). 

Given the inclusion order of the sets, the family of all $f$-invariant sets and reducing subspaces for $C^{*}(T_{1},T_{2},\cdots,T_{k})$ can be regarded as complete lattices. 
The next theorem applies to any $f$ that satisfies the bounded condition. 

\begin{theorem} \label{Inj}
There exists an injective, order-preserving map from 

\begin{equation*}
\{K \subseteq X \mid \text{$K$ is a $f$-invariant set}\} 
\end{equation*}

to 

\begin{equation*}
\{\M \subseteq \H \mid \text{$\M$ is a reducing subspace for $C^{*}(T_{1},T_{2},\cdots,T_{k})$}\}. 
\end{equation*}

\end{theorem}

\begin{proof}
For a trivial $f$-invariant set $K=\emptyset \subseteq X$, let $\H_{K}=\{0\} \subseteq \H$ as a trivial reducing subspace. 

Let $K (\neq \emptyset) \subseteq X$ be a $f$-invariant set and define $\H_{K} \subseteq \H$ by: 

\begin{align*}
\H_{K}=\overline{\Span\{e_{x} \mid x \in K}\}. 
\end{align*}

For two $f$-invariant sets $K \neq K' \subseteq X$, by the above definition, it holds that $\H_{K} \neq \H_{K'}$. 
If $K \subseteq K'$, then $\H_{K} \subseteq \H_{K'}$. 

We will see that $\H_{K}$ is a reducing subspace for $C^{*}(T_{1},T_{2},\cdots,T_{k})$. 
Let $e_{x} \in \H_{K}$. 
By the definition of $T_{i}$, 

\begin{equation*}
T_{i}e_{x}, T_{i}^{*}e_{x} \in \Span\{e_{y} \in \H \mid y=f(x) \text{ or } y \in f^{-1}(x)\} 
\end{equation*}

for $1 \leq i \leq k$. 
Since $K$ is a $f$-invariant set, $f(x) \in K$ and $f^{-1}(x) \subseteq K$. 
Therefore, $C^{*}(T_{1},T_{2},\cdots,T_{k})e_{x} \subseteq \H_{K}$. 
It follows that $C^{*}(T_{1},T_{2},\cdots,T_{k})\H_{K} \subseteq \H_{K}$, and hence, $\H_{K}$ is a reducing subspace for $C^{*}(T_{1},T_{2},\cdots,T_{k})$. 
\end{proof}


\section{Uniqueness condition}


In this section, we introduce the concept of the uniqueness condition for analyzing a loop in a dynamical system. 
For $1 \leq j \leq k$, we let $f_{j}$ denote $f|_{X_{j}}$. 
For $x \in X$, if there exists $m \in \N$ such that $f^{m}(x)=x$, then $x$ is called a periodic point. 
In particular, when $m=1$, $f(x)=x$ is called a fixed point of $f$. 
A loop in the dynamical system $(X,f)$ is a finite sequence $\{x_{i}\}_{i=1}^{m} \subseteq X$ such that $f(x_{i})=x_{i+1}$ for $1 \leq i < m$ and $f(x_{m})=x_{1}$. 

Fix $x \in X$ and let $m \in \N$. 
Consider the set of indices, $\{i_{j}\}_{j=1}^{m} \subseteq \{1,2,\cdots,k\}$ satisfying $f^{j-1}(x) \in X_{i_{j}}$, the $m$-tuple 

\begin{align*}
I=(i_{1},i_{2},\cdots,i_{m}) \in \{1,2,\cdots,k\}^{m} 
\end{align*}

is uniquely determined for $x$. 

For any tuple $I=(i_{1},i_{2},\cdots,i_{m}) \in \{1,2,\cdots,k\}^{*}$, we use the notation $f_{I}$ and $T_{I}$ for the map $f_{i_{m}} \circ f_{i_{m-1}} \circ \cdots \circ f_{i_{1}}$ (if defined) and the bounded linear operator $T_{i_{m}} T_{i_{m-1}} \cdots T_{i_{1}}$. 
The following are then equivalent: 

\begin{enumerate}[(i)]
\item	$f$ has a periodic point; 
\item	$f_{I}$ has a fixed point for some $I \in \{1,2,\cdots,k\}^{*}$. 
\end{enumerate}

In this regard, we define the uniqueness condition for $f$. 

\begin{definition}[Uniqueness condition]
When $f_{I}$ (if defined) has at most one fixed point for every $I=(i_{1},i_{2},\cdots,i_{m}) \in \{1,2,\cdots,k\}^{*}$, 
we say that $f$ satisfies the uniqueness condition. 
\end{definition}

The following maps satisfy the uniqueness condition. 

\begin{example} \label{ExaUni}

\begin{enumerate}[(i)]

\item	($q x{+}d$-function) 
Let $q,d \in \N$ be odd numbers. 
We define the $q x{+}d$-function $f_{q,d}:\N \to \N$ as follows: 

\begin{equation*}
f_{q,d}(n)= 
\begin{cases}
q n+d,	&	n:\text{odd}, \\ 
n/2,		&	n:\text{even}. 
\end{cases}
\end{equation*}

By setting $X_{1}=\{n \in \N \mid n:\text{odd}\}$ and $X_{2}=\{n \in \N \mid n:\text{even}\}$, $f_{q,d}$ satisfies the bounded condition. 
Since $q n{+}d$ and $n/2$ are linear functions and $d \neq 0$, $(f_{q,d})_{I}(x)$ (if defined) is a linear function for every $I \in \{1,2\}^{*}$ satisfying $(f_{q,d})_{I}(x) \not\equiv x$ ($x \in \N$). 
Therefore, the equation $(f_{q,d})_{I}(x)=x$ has at most one solution where $x$ is a natural number. 
Consequently, $f_{q,d}$ satisfies the uniqueness condition. 

\item	($(\a,\b)$-system) 
Let $k \in \N_{>1}$ and $\a=(a_{1},a_{2},\cdots,a_{k-1}), \b=(b_{1},b_{2},\cdots,b_{k-1}) \in \N^{k-1}$. 
The function $f_{\a,\b}:\N \to \N$ is defined by: 

\begin{equation*}
f_{\a,\b}(n)= 
\begin{cases}
a_{i}n+b_{i},	&	n \equiv i \pmod{k}, \ 1 \leq i \leq k-1, \\ 
n/k,			&	n \equiv 0 \pmod{k}. 
\end{cases}
\end{equation*}

By setting $X_{i}=\{n \in \N \mid n \equiv i \pmod{k}\}$ where $1 \leq i \leq k$, $f_{\a,\b}$ satisfies the bounded condition. 
Since $a_{i}x+b_{i}$ and $x/k$ are linear functions and $b_{i} \neq 0$ for every $1 \leq i \leq k$, $(f_{\a,\b})_{I}(x)$ (if defined) is a linear function for every $I \in \{1,2,\cdots,k\}^{*}$ satisfying $(f_{a,b})_{I}(x) \not\equiv x$ ($x \in \N$). 
Therefore, the equation $(f_{\a,\b})_{I}(x)=x$ has at most one solution where $x$ is a natural number. 
Consequently, $f_{\a,\b}$ satisfies the uniqueness condition. 
\end{enumerate}

\end{example}

If $f$ satisfies the uniqueness condition and the dynamical system $(X,f)$ has a loop, then $f$ satisfies the separating condition for a point. 
According to Theorems \ref{IrrMin} and \ref{MinIrr}, such a dynamical system $(X,f)$ is minimal if and only if the associated $C^{*}$-algebra is irreducible. 
We prove the relationship between the uniqueness and separating conditions for maps in the following proposition. 

\begin{proposition} \label{UniFix}
Suppose that $f$ satisfies the uniqueness condition. 
For every $I \in \{1,2,\cdots,k\}^{*}$ and $m \in \N$, if $f_{I}$ and $f_{I}^{m}$ are defined, then the following are equivalent: 

\begin{enumerate}[(i)]
\item	$f_{I}$ has a fixed point $x$; 
\item	$f_{I}^{m}$ has a fixed point $x$. 
\end{enumerate}

In particular, if $f$ has a periodic point $x$, then there exists an aperiodic tuple $I \in \{1,2,\cdots,k\}^{*}$ satisfying that $f_{I}$ has a fixed point $x$. 
\end{proposition}

\begin{proof}
If $x \in X$ is a fixed point of $f_{I}$, then $f_{I}^{m}(x)=x$. 
Hence, $(i) \implies (ii)$ holds. 

Assume that $x \in X$ is a fixed point of $f_{I}^{m}$. 
It follows that 

\begin{align*}
f_{I}^{m}(f_{I}(x))=f_{I}(f_{I}^{m}(x))=f_{I}(x). 
\end{align*}

Thus, $f_{I}(x)$ is a fixed point of $f_{I}^{m}$. 
By the uniqueness condition of $f$, $f_{I}^{m}$ has at most one fixed point, and hence $f_{I}(x)=x$. 
Therefore, $x$ is a fixed point of $f_{I}$ and $(ii) \implies (i)$. 

Let $f$ has a periodic point $x$ and $m$ be a minimum positive integer such that $f^{m}(x)=x$. 
There uniquely exists $I=(i_{1},i_{2},\cdots,i_{m}) \in \{1,2,\cdots,k\}^{*}$ satisfying that $f^{n-1}(x) \in X_{i_{n}}$ for every $1 \leq n \leq m$. 
We will show that a contradiction arises if we assume that $I$ is periodic. 

Assume that $0<l<m$ and that $J=(j_{1},j_{2},\cdots,j_{l}) \in \{1,2,\cdots,k\}^{*}$ such that 

\begin{equation*}
I=(j_{1},j_{2},\cdots,j_{l},j_{1},j_{2},\cdots,j_{l},\cdots,j_{1},j_{2},\cdots,j_{l}). 
\end{equation*}

It follows that
 
\begin{align*}
x=f^{m}(x)=f_{I}(x)=f_{J}^{m/l}(x). 
\end{align*}

Hence, $f_{J}(x)=x$. 
However, this contradicts the assumption that $m$ is the minimum positive integer such that $f^{m}(x)=x$. 
Therefore, $I$ is aperiodic. 
\end{proof}

\begin{corollary}
Suppose that $x \in X$ is a periodic point of $f$. 
If $f$ satisfies the uniqueness condition, then it also satisfies the separating condition for $x$. 
\end{corollary}

\begin{proof}
This is a consequence of Definition \ref{SepCon} and Proposition \ref{UniFix}. 
\end{proof}

For the Collatz map, the tuple of parities for every loop is aperiodic. 

\begin{corollary}
Let $q,d \in \N$ be odd numbers and let $f_{q,d}$ be the $q x{+}d$-function. 
Suppose that $x \in \N$ is a periodic point of $f_{q,d}$ and $m \in \N$ is the minimum positive integer such that $f_{q,d}^{m}(x)=x$. 
Define the set $\{i_{j}\}_{j=1}^{m} \subseteq \{1,2\}$ by: 

\begin{equation*}
i_{j}= 
\begin{cases}
1,	&	f_{q,d}^{j-1}(x):\text{odd}, \\ 
2,	&	f_{q,d}^{j-1}(x):\text{even}. 
\end{cases}
\end{equation*}

In the above settings, $I=(i_{1},i_{2},\cdots,i_{m})$ is aperiodic. 
In particular, if $n \in \N$ is a periodic point for the Collatz map $f$ and $m \in \N$ is the minimum positive integer such that $f^{m}(n)=n$, then the $m$-tuple 

\begin{equation*}
((-1)^{n},(-1)^{f(n)},\cdots,(-1)^{f^{m-1}(n)}) 
\end{equation*}

is aperiodic. 
\end{corollary}

\begin{proof}
When $q = 3$ and $d = 1$, $f_{q,d}$ is equal to the Collatz map and it satisfies the uniqueness condition. 
Suppose $n \in \N$ is a periodic point of $f_{3,1}$ and $m \in \N$ is the minimum positive integer such that $f_{3,1}^{m}(n)=n$. 
Define $I=(i_{1},i_{2},\cdots,i_{m})$ as the statement and $J=((-1)^{n},(-1)^{f_{3,1}(n)},\cdots,(-1)^{f_{3,1}^{m-1}(n)})$. 
From the above proposition, $I$ is aperiodic. 
Since, for $1 \leq j < m$, $i_{j}=1$ if and only if $f_{3,1}^{j-1}(n)$ is odd, it follows that $f_{3,1}^{j-1}(n)=1$ if and only if $(-1)^{f_{3,1}^{j-1}(n)}=-1$. 
Similarly, $i_{j}=2$ if and only if $(-1)^{f_{3,1}^{j-1}(n)}=1$. 
This implies that $I$ is aperiodic if and only if $J$ is aperiodic. 
Hence, $J$ is aperiodic. 
\end{proof}

The next result applies to any map that satisfies the bounded condition,  regardless of whether the uniqueness condition is assumed. 

\begin{lemma}
Let $I \in \{1,2,\cdots,k\}^{*}$. 
If $f_{I}$ has a fixed point, then $T_{I}$ has a non-zero fixed point. 
\end{lemma}

\begin{proof}
Let $x$ be the fixed point of $f_{I}$. 
Then, 
\begin{equation*}
0 \neq T_{I}e_{x}=e_{f_{I}(x)}=e_{x}. 
\end{equation*}
Thus, $T_{I}$ has a non-zero fixed point. 
\end{proof}

The uniqueness condition of a map leads that the converse of the above lemma holds. 

\begin{theorem}
Assume that $f$ satisfies the uniqueness condition and that $I \in \{1,2,\cdots,k\}^{*}$ satisfies the condition that $T_{I} \neq 0$. 
If $T_{I}$ has a non-zero fixed point, then $f_{I}$ has a fixed point. 
Furthermore, if $x \in X$ is the fixed point of $f_{I}$, then the set of all fixed points of $T_{I}$ is $\{\alpha e_{x} \mid \alpha \in \C\}$. 
\end{theorem}

\begin{proof}
Let $a \in \H\backslash\{0\}$ be a fixed point of $T_{I}$. 
Since $\norm{T_{I}} \leq 1$, it holds that $\norm{T_{I}^{*}a} \leq \norm{a}$, and so, 

\begin{align*}
&\norm{T_{I}^{*}a-a}^2 
=\norm{T_{I}^{*}a}^2-\innpr{a}{T_{I}a}-\innpr{T_{I}a}{a}+\innpr{a}{a}=\norm{T_{I}^{*}a}^2-\innpr{a}{a}-\innpr{a}{a}+\innpr{a}{a}\\ 
&=\norm{T_{I}^{*}a}^2-\innpr{a}{a}=\norm{T_{I}^{*}a}^2-\norm{a}^2 \leq \norm{a}^2-\norm{a}^2=0. 
\end{align*}

We obtain that $T_{I}^{*}a=a$. 
Let $x \in X$ satisfy $\innpr{a}{e_{x}} \neq 0$. 
Then, 

\begin{align*}
0 \neq \innpr{a}{e_{x}}=\innpr{(T_{I}^{*})^{m} a}{e_{x}}=\innpr{a}{T_{I}^{m} e_{x}} 
\end{align*}

for every $m \in \N$. 
It follows that $T_{I}^{m} e_{x} \neq 0$ and $T_{I}^{m} e_{x}=e_{f_{I}^{m}(x)}$. 
Since 

\begin{equation*}
\norm{a}^2=\sum_{y \in X}\abs{\innpr{a}{e_{y}}}^2 < +\infty, 
\end{equation*}

the set $\{e_{f_{I}^{m}(x)}\}_{m \in \N}$ is finite. 
Hence, there exist $M < N \in \N$ such that $f_{I}^{M}(x)=f_{I}^{N}(x)$. 
Let $m=N-M$. 
Then, $f_{I}^{m}(f_{I}^{M}(x))=f_{I}^{N}(x)=f_{I}^{M}(x)$. 
Thus, $f_{I}^{m}$ has a fixed point. 
Similarly, $f_{I}^{m}(f_{I}^{M+1}(x))=f_{I}^{m+1}(f_{I}^{M}(x))=f_{I}(f_{I}^{M}(x))=f_{I}^{M+1}(x)$. 
By the uniqueness condition of $f$, 

\begin{equation*}
f_{I}^{M}(f_{I}(x))=f_{I}^{M+1}(x)=f_{I}^{M}(x). 
\end{equation*}

Since $f_{I}^{M}$ is injective, $f_{I}(x)=x$. 
Hence, $x$ is the fixed point of $f_{I}$. 

Conversely, assume that $x \in X$ is a fixed point of $f_{I}$. 
Then, $e_{x}$ is a fixed point of $T_{I}$. 
Let $a \in \H\backslash\{0\}$ be a fixed point of $T_{I}$. 
If there is $y \in X\backslash\{x\}$ satisfying that $\innpr{a}{e_{y}} \neq 0$, then it holds that 

\begin{align*}
0 \neq \innpr{a}{e_{y}}=\innpr{(T_{I}^{*})^{m}a}{e_{y}}=\innpr{a}{T_{I}^{m}e_{y}}
=\langle a, e_{f_{I}^{m}(y)} \rangle 
\end{align*}

for every $m \in \N$. 
And similarly, it also holds that $f_{I}(y)=y$. 
Since $f_{I}(x)=x$, by using the assumption of the uniqueness condition of $f$, we obtain that $x=y$ which contradicts to the assumption that $x \neq y$. 
Thus, $a \in \Span\{e_{x}\}=\{\alpha e_{x} \mid \alpha \in \C\}$. 
\end{proof}


\section{Totally uniqueness condition}


Fix $x \in X$. 
For every $j \in \N$, there is a unique $i_{j} \in \{1,2,\cdots,k\}$ such that $f^{j-1}(x) \in X_{i_{j}}$. 
We define an infinite sequence $\hat{x} \in \{1,2,\cdots,k\}^{\N}$ as follows: 

\begin{equation*}
\hat{x}=(i_{j})_{j \in \N}=(i_{1},i_{2},i_{3},\cdots). 
\end{equation*}

\begin{definition}[Totally uniqueness condition]
When the map $X \ni x \mapsto \hat{x} \in \{1,2,\cdots,k\}^{\N}$ is injective, we say that $f$ satisfies the totally uniqueness condition. 
\end{definition}

The totally uniqueness condition is stronger than the uniqueness condition. 

\begin{proposition}
If $f$ satisfies the totally uniqueness condition, then $f$ satisfies the uniqueness condition. 
\end{proposition}

\begin{proof}
Suppose $I=(i_{1},i_{2},\cdots,i_{m}) \in \{1,2,\cdots,k\}^{*}$ satisfies that there are $x,y \in X$ such that $f_{I}(x)=x$ and $f_{I}(y)=y$. 
Then, it follows that $\hat{x}=\hat{y}$. 
By the totally uniqueness condition of $f$, it holds that $x=y$. 
Therefore, $f$ satisfies the uniqueness condition. 
\end{proof}

When $f$ satisfies the totally uniqueness condition, we will show that the order-preserving map in Theorem \ref{Inj} is bijective. 
Therefore, for such an $f$, there is a one-to-one correspondence between the families of all $f$-invariant sets and all reducing subspaces for $C^{*}(T_{1},T_{2},\cdots,T_{k})$. 

\begin{theorem} \label{Bij}
If $f$ satisfies the totally uniqueness condition, then the injective, order-preserving map in Theorem \ref{Inj} is bijective. 
\end{theorem}

\begin{proof}
Suppose that $f$ satisfies the totally uniqueness condition. 
Let $\M \subseteq \H$ be a reducing subspace for $C^{*}(T_{1},T_{2},\cdots,T_{k})$. 
If $\M=\{0\}$, then it is the image of $\emptyset \subseteq X$ by the map. 

Assume that $\M \neq \{0\}$. 
Let $a \in \M\backslash\{0\}$ and $x \in X$ satisfy $\innpr{a}{e_{x}} \neq 0$. 
Let $\hat{x}=(i_{j})_{j \in \N} \in \{1,2,\cdots,k\}^{\N}$. 
By the totally uniqueness condition, for every $y \in X\backslash\{x\}$, there exists $j \in \N$ such that $f^{j-1}(x) \in X_{i_{j}}$ and $f^{j-1}(y) \notin X_{i_{j}}$. 
For $m \in \N$, define

\begin{equation*}
P_{m}=T_{i_{1}}^{*} T_{i_{2}}^{*} \cdots T_{i_{m}}^{*} T_{i_{m}} \cdots T_{i_{2}} T_{i_{1}} \in C^{*}(T_{1},T_{2},\cdots,T_{k}). 
\end{equation*}

Then, $P_{j}e_{x}=e_{x}$ and $P_{j}e_{y}=0$. 
This implies that 

\begin{equation*}
P_{m}a = P_{m}(\sum_{y \in X}\innpr{a}{e_{y}}e_{y}) \to \innpr{a}{e_{x}}e_{x} \text{ } (m \to \infty). 
\end{equation*}

Therefore, $e_{x} \in \overline{C^{*}(T_{1},T_{2},\cdots,T_{k})\M} \subseteq \M$,  and hence 

\begin{equation*}
\overline{\Span\{e_{x} \mid x \in X, \innpr{a}{e_{x}} \neq 0 \text{ for some }a \in \M\}} \subseteq \M. 
\end{equation*}

The converse of this inclusion relationship is also true. 
Thus, 

\begin{equation*}
\M=\overline{\Span\{e_{x} \mid x \in X, \innpr{a}{e_{x}} \neq 0 \text{ for some }a \in \M\}}. 
\end{equation*}

Let $K=\{x \in X \mid \innpr{a}{e_{x}} \neq 0 \text{ for some }a \in \M\}$. 
It is easy to check that $K$ is a $f$-invariant set, because $\M$ is a reducing subspace for $C^{*}(T_{1},T_{2},\cdots,T_{k})$. 
This implies $\M=\H_{K}$. 
Therefore, the map in Theorem \ref{Inj} is surjective. 
\end{proof}

Let us consider the $(\a,\b)$-system described in Example \ref{ExaUni} $(ii)$. 
If $k=2$, $\a=(a_{1})=(q)$, $\b=(b_{1})=(d)$ where $q,d \in \N$ are arbitrary odd numbers, then $f_{\a,\b}$ is the $q x{+}d$-function. 
At this time, it follows that $a_{1}=q$ is relatively prime to $k$ and, $\{n,f_{q,d}(n)\}$ includes a multiple of $2$ for every $n \in \N$. 

For $k \geq 2$, in what follows, we assume the following conditions: 

\begin{enumerate}[(i)]
\item	$a_{i}$ is relatively prime to $k$ for every $1 \leq i \leq k$. 
\item	$\{n,f_{\a,\b}(n),\cdots,f_{\a,\b}^{k-1}(n)\}$ includes a multiple of $k$ for every $n \in \N$. 
\end{enumerate}

Let $f=f_{\a,\b}$ and set $X_{i}=\{n \in \N \mid n \equiv i \pmod{k}\}$ where $1 \leq i \leq k$. 
Next, let us consider the ring 

\begin{equation*}
X=\projlim_{n\to\infty}\Z/k^{n}\Z. 
\end{equation*}

By considering the discrete topology on $\Z/k^{n}\Z$, we can assume that $\prod_{n \in \N}\Z/k^{n}\Z$ is a topological space with the product topology and we define the relative topology on $X$ of the product space. 
We let $\pi_{n}:X\to\Z/k^{n}\Z$ denote the canonical ring homomorphism. 
The map $f:\N\to\N$ extends naturally to a continuous map $f:X\to X$, namely 

\begin{equation*}
f(x)= 
\begin{cases}
a_{i}x+b_{i}	&	\pi_1(x) \equiv i \pmod k,\ 1 \leq i \leq k-1, \\ 
x/k			&	\pi_1(x) \equiv 0 \pmod k. 
\end{cases}
\end{equation*}

Since $a_{i}$ is relatively prime to $k$ for every $i=1,2,\dots,{k-1}$, one can verify that $f$ becomes a local homeomorphism. 

For $x\in X$, we define $\hat{x} \in (\Z/k\Z)^{\N}$ by 

\begin{equation*}
\hat{x}(j)=\pi_{1}(f^{j-1}(x)) 
\end{equation*}

where $j \in \N$. 

\begin{lemma}
Let $x,y\in X$ and let $j \in \N$. 
Suppose $i:=\pi_{1}(x)=\pi_{1}(y)$ and $\pi_{j}(f(x))=\pi_{j}(f(y))$. 

\begin{enumerate}[(i)]
\item	When $i \neq 0$, we have $\pi_{j}(x)=\pi_{j}(y)$. 
\item	When $i=0$, we have $\pi_{j+1}(x)=\pi_{j+1}(y)$. 
\end{enumerate}

\end{lemma}

\begin{proof}
$(i)$\ 
One has $f(x)=a_{i}x+b_{i}$ and $f(y)=a_{i}y+b_{i}$ by the definition of $f$. 
Since $a_{i}$ and $k$ are coprime, $\pi_{j}(f(x))=\pi_{j}(f(y))$ implies $\pi_{j}(x)=\pi_{j}(y)$. 

$(ii)$\ 
We have $k f(x)=x$ and $k f(y)=y$ by the definition of $f$. 
This, together with $\pi_{j}(f(x))=\pi_{j}(f(y))$, 
implies $\pi_{j+1}(x)=\pi_{j+1}(y)$. 
\end{proof}

\begin{lemma}\label{MapInj}
The map $x \mapsto \hat{x}$, where $x \in X$, is injective. 
\end{lemma}

\begin{proof}
Assume that there exist $x,y\in X$ such that $\hat{x}=\hat{y}$. 
For any $j \in\N$, 

\begin{equation*}
\pi_{1}(f^{j-1}(x))=\hat{x}(j)=\hat{y}(j)=\pi_{1}(f^{j-1}(y)) 
\end{equation*}

holds. 
Take $m\in\N$ arbitrarily. 
The cardinality of the set 

\begin{equation*}
\left\{j\in\{1,\dots,m k\}\mid \pi_{1}(f^{j-1}(x))=0\right\} 
\end{equation*}

is not less than $m$, 
because $\{n,f(n),f^{2}(n),\cdots,f^{k-1}(n)\}$ includes a multiple of $k$ for every natural number $n$. 
Hence, starting from $\pi_{1}(f^{m k}(x))=\pi_1(f^{m k}(y))$ and 
applying the lemma above inductively, 
we obtain $\pi_{1+l}(x)=\pi_{1+l}(y)$ for some $l \geq m$. 
As $m$ was arbitrary, one can conclude $x=y$. 
\end{proof}

\begin{proposition} \label{fqdSatTUC}
In the above settings, $f_{\a,\b}:\N \to \N$ satisfies the totally uniqueness condition. 
In particular, for every odd number $q,d \in \N$, the $q x{+}d$-function satisfies the totally uniqueness condition. 
\end{proposition}

\begin{proof}
By identifying sets $\{1,2,\cdots,k\}$ and $\Z/k\Z$, the map in Lemma \ref{MapInj} can be regarded as an extension of the map $n \mapsto \hat{n}$, where $n \in \N$. 
Since the former map is injective, so is the latter. 
\end{proof}

The next corollary follows from Theorem \ref{Bij} and the previous proposition. 

\begin{corollary} \label{fqdIrrMin}
Let $q,d \in \N$ be odd numbers and $f_{q,d}$ be the $q x{+}d$-function. 
Set $X_{1}=\{n \in \N \mid n:\text{odd}\}$ and $X_{2}=\{n \in \N \mid n:\text{even}\}$. 
Let $C^{*}(T_{1},T_{2})$ be the associated $C^{*}$-algebra. 
Then, the following  are equivalent: 

\begin{enumerate}[(i)]
\item	$C^{*}(T_{1},T_{2})$ is irreducible; 
\item	$(\N,f_{q,d})$ is minimal. 
\end{enumerate}

\end{corollary}

\begin{proof}
According to Proposition \ref{fqdSatTUC}, $f_{q,d}$ satisfies the totally uniqueness condition. 
This implies that there exists a one-to-one correspondence between the families of all $f_{q,d}$-invariant sets and all reducing subspaces for $C^{*}(T_{1},T_{2})$. 
In particular, $C^{*}(T_{1},T_{2})$ is irreducible if and only if $(\N,f_{q,d})$ is minimal. 
\end{proof}


\section{Symbolic representation}


Thus far, we have discussed dynamical systems that lack topological structure. 
In this section, we explore the possibility of incorporating topological structure to develop previous results further. 
Here, we open the possibility for topological analysis by linking the method of symbolic dynamics - a standard technique in the study of topological dynamical systems - with dynamical systems on discrete phase spaces. 
The totally uniqueness condition introduced in the previous section plays a crucial role in this endeavor. 

Before we delve into the main topic, let us review the theory of symbolic dynamical systems in the context of conventional phase dynamical systems. 
For standard references in this field, we recommand the textbook \cite{AH94NHML}. 
Much of the knowledge regarding topological dynamical systems in this section is based on this reference. 

Consider the general setting of a discrete-time dynamical system, which is defined by a compact metric space $X$ and a homeomorphism $f$ on it. 
The method of analyzing the properties of a given dynamical system by relating it to a system involving abstract sequences of symbols has long been employed in the study of individual dynamical systems. 
Much of this research was based on the goal of facilitating the analysis of the original phenomenon by linking smooth or continuous changes on differentiable manifolds or topological spaces to discrete objects, namely infinite sequences of abstract symbols. 
Markov partitions (a special partitioning method for dynamical systems) are an extremely important tool for carrying out this type of analysis. 

For a homeomorphism $f$ on a compact metric space $X$, a Markov partition is a finite cover $\{R_{1},R_{2},\cdots,R_{k}\}$ of $X$ such that 

\begin{enumerate}[(i)]
\item	each $R_{i}$ is a proper rectangle, 
\item	$\text{int}(R_{i}) \cap \text{int}(R_{j}) = \emptyset$ for $i \neq j$, 
\item	let $x \in \text{int}(R_{i}) \cap f^{-1}(\text{int}(R_{j}))$, then 

\begin{equation*}
f(V^{s}(x,R_{i}) \subseteq V^{s}(f(x),R_{j}) \text{ and } f(V^{u}(x,R_{i}) \supseteq V^{u}(f(x),R_{j}). 
\end{equation*}

\end{enumerate}

For a Markov partition $\{R_{1},R_{2},\cdots,R_{k}\}$, define the transition matrix $A=(A_{i,j})$ by 

\begin{equation*}
A(i,j)= 
\begin{cases}
1,	&	\text{int}(R_{i}) \cap f^{-1}(\text{int}(R_{j})) \neq \emptyset, \\ 
0,	&	\text{int}(R_{i}) \cap f^{-1}(\text{int}(R_{j})) = \emptyset. 
\end{cases}
\end{equation*}

The set of all bi-infinite sequences $x=(\cdots,x_{-1},x_{0},x_{1},\cdots)$ where $x_{n} \in \{1,2,\cdots,k\}$ becomes a compact metric space if we define the distance between $x,y \in X$ by $\dist(x,y)=0$ if $x=y$ and $\dist(x,y)=2^{-m}$ if $x \neq y$ where $m$ is the largest integer such that $x_{n}=y_{n}$ for all $n$ with $|n|<m$. 
Assume that there exists a Markov partition for $f$ and let $A$ be the associated transition matrix. 
Let $\Sigma_{A}$ be the compact subset of the above set defined by: 

\begin{align*}
\Sigma_{A}=\{x=(x_{i}) \mid A(x_{i},x_{i+1})=1 \text{ for } i \in \Z\} 
\end{align*}

and $\sigma:\Sigma_{A} \to \Sigma_{A}$ be the shift map defined by $\sigma((x_{i}))=(x_{i+1})$ as usual. 
Then, there exists a continuous surjection $\pi:\Sigma_{A} \to X$ with $\pi \circ \sigma = f \circ \pi$ such that $\pi^{-1}(x)$ has at most $d$ points for every $x \in X$, for a certain positive integer $d$. 
This implies that the homeomorphism $f$ behaves as the shift map $\sigma$. 
In this respect, the existence of a Markov partition is of paramount importance in dynamical systems theory. 
This correspondence to such desirable dynamical systems is called the symbolic representation of the original dynamical system. 
Simply put, the existence of a Markov partition guarantees the existence of a faithful representation from the dynamical system to the symbolical dynamics. 

On the other hand, when the initial assumption of a compact metric space and invertible map on it does not hold, various situations arise that these theories cannot cover. 
Our dynamical system also belongs to this class of cases that standard topological dynamical system theory cannot handle. 
When discussing settings that deviate from such standard cases, it is essential to carefully examine proofs starting from elementary assertions. 

From now on, we discuss a symbolic representation for a dynamical system with discrete phase space. 
We will write $\hat{X}$ as the set of all $\hat{x}$ where $x \in X$. 
Let $x \in X$ and $\hat{x}=(\hat{x}(1),\hat{x}(2),\hat{x}(3),\cdots) \in \{1,2,\cdots,k\}^{\N}$. 
By the definition of $\hat{x}$, it follows that $\widehat{f(x)}=(\hat{x}(2),\hat{x}(3),\hat{x}(4),\cdots)$. 

Let $\sigma:\hat{X} \to \hat{X}$ defined by: 

\begin{equation*}
\sigma((\hat{x}(j))_{j \in \N})=(\hat{x}(j+1))_{j \in \N}, 
\end{equation*}

and set $\hat{X}_{i}=\{\hat{x} \in \{1,2,\cdots,k\}^{\N} \mid x \in X_{i}\}$ where $1 \leq i \leq k$. 
Then, $\sigma$ satisfies the bounded condition and we call $(\hat{X},\sigma)$ the symbolic dynamics of $(X,f)$. 
$\sigma$ is called the shift map on $\hat{X}$. 

Define a map $\pi : X \to \hat{X}$ by: 

\begin{align*}
X \ni x \mapsto \pi(x)=\hat{x} \in \hat{X}. 
\end{align*}

This canonical map satisfies the following property. 

\begin{lemma} \label{MapSurHom}
$\pi$ is a surjective map from $X$ to $\hat{X}$ such that $\pi \circ f=\sigma \circ \pi$ and $\pi(X_{i})=\hat{X}_{i}$ for every $1 \leq i \leq k$. 
\end{lemma}

For $(X,f)$, a dynamical system with discrete phase space, $f$ does not always satisfy the totally uniqueness condition when $f$ satisfies the bounded condition. 
On the other hand, the following statement holds for every $(\hat{X},\sigma)$. 

\begin{proposition}\label{HatSatTUC}
For a symbolic dynamics $(\hat{X},\sigma)$, the shift map $\sigma$ satisfies the totally uniqueness condition. 
\end{proposition}

\begin{proof}
We will check that $\hat{f}^{j-1}(\hat{x}) \in \hat{X}_{\hat{x}(j)}$ for every $j \in \N$ and $\hat{x} \in \hat{X}$. 

Since $\widehat{f(x)}=\sigma(\hat{x})$, 

\begin{equation*}
\widehat{f^{j-1}(x)}=\sigma^{j-1}(\hat{x}). 
\end{equation*}

This implies that $f^{j-1}(x) \in X_{i}$ if and only if $\sigma^{j-1}(\hat{x}) \in \hat{X}_{i}$ for every $1 \leq i \leq k$. 
On the other hand, $\sigma^{j-1}(\hat{x}) \in \hat{X}_{i}$ if and only if $i=\hat{x}(j)$. 
Therefore, 

\begin{equation*}
\sigma^{j-1}(\hat{x}) \in \hat{X}_{\hat{x}(j)}. 
\end{equation*}

Let $\hat{x},\hat{y} \in \hat{X}$. 
If $\hat{x} \neq \hat{y}$, then there exists $j \in \N$ satisfying that $\hat{x}(j) \neq \hat{y}(j)$. 
Then, $\sigma^{j-1}(\hat{x}) \in \hat{X}_{\hat{x}(j)}$ and $\sigma^{j-1}(\hat{y}) \notin \hat{X}_{\hat{x}(j)}$. 
Thus, $\sigma$ satisfies the totally uniqueness condition. 
\end{proof}

The previous proposition and Theorem \ref{Bij} imply the following corollary. 

\begin{corollary}
For every $(X,f)$, the family of $\sigma$-invariant sets for the dynamical system $(\hat{X},\sigma)$ perfectly aligned with the family of reducing subspaces for the associated $C^{*}$-algebras as complete lattices. 
\end{corollary}

The following proposition is a direct consequence of the definition of the totally uniqueness condition for maps. 

\begin{proposition}
$\pi$ is bijective if and only if $f$ satisfies the totally uniqueness condition. 
\end{proposition}

\begin{remark}
A Markov partition $\{R_{1},R_{2},\cdots,R_{k}\}$ for a homeomorphism $f:X \to X$ of a compact metric space, gives rise to a symbolic dynamics, denoted by the map $\sigma:\Sigma_{A} \to \Sigma_{A}$ and a continuous surjection $\pi:\Sigma_{A} \to X$ such that $f \circ \pi=\pi \circ \sigma$ (not $\pi \circ f =\sigma \circ \pi$). 

The construction of $\pi$ is as follows: 
for $a=(\cdots, a_{-1}, a_{0}, a_{1}, \cdots) \in \Sigma_{A}$, let $K(a)$ be the set $\cap_{n \in \Z}f^{-j}(R_{a_{j}})$. 
According to the compactness of $X$, $K(a)$ has a single point. 
Then, let this point be the image of $a$ by $\pi$. 

When $f$ is not invertible, the definition of $K(a)$ in the above construction can be considered as follows: 

\begin{align*}
K(a)=\cap_{j \in \N}f^{-j}(R_{a_{j}}), 
\end{align*}

where $a=(a_{1},a_{2},a_{3},\cdots) \in \prod_{n \in \N}\{1,2,\cdots,k\}$. 

For a dynamical system $(X,f)$ with discrete phase space, assume that $f$ satisfies the bounded condition. 
Define the transition matrix $A=A(i,j)$ by 

\begin{align*}
A(i,j)= 
\begin{cases}
1,	&	X_{i} \cap f^{-1}(X_{j}) \neq \emptyset \\ 
0,	&	X_{i} \cap f^{-1}(X_{j}) = \emptyset, 
\end{cases}
\end{align*}

and the set $\Sigma_{A}$ by 

\begin{align*}
\Sigma_{A}=\{x=(x_{i}) \mid A(x_{i},x_{i+1})=1 \text{ for } i \in \N\}. 
\end{align*}

At this time, $K(a) \subseteq X$ is not always a set has only a single point. 
Let us prove this using the Collatz map as an example. 

When $X=\N$ and $f$ is the Collatz map, there exists a basic partition $X_{1}=\{n \in \N \mid n \text{: odd}\}$ and $X_{2}=\{n \in \N \mid n \text{: even}\}$ such that $f$ satisfies the bounded and totally uniqueness condition. 
Then, the transition matrix $A$ is $\begin{pmatrix}0 & 1 \\ 1 & 1\end{pmatrix}$. 
For a infinite sequence $a=(2,2,2,\cdots) \in \Sigma_{A}$, $x \in K(a)$ if and only if $f(x)^{j} \in X_{2}$ for every $j \in \N$. 
On the other hand, for every $n \in \N$, $f^{m}(n) \in X_{1}$ for some $m \in \N$. 
Hence, $K(a)=\emptyset$. 
\end{remark}


\section*{Acknowledgements}


I would like to thank Professor Hiroki Matui (Chiba University), my thesis supervisor, for helpful discussions in the seminar of the study.


\bibliographystyle{plain}

\bibliography{bib}


\end{document}